\documentclass[10pt]{amsart}

\usepackage{amsmath}
\usepackage{amsthm}
\usepackage{amssymb}

\usepackage{diagrams}

\theoremstyle{plain}
\newtheorem{dfn}[subsection]{Definition}
\newtheorem{thm}[subsection]{Theorem}
\newtheorem{prp}[subsection]{Proposition}
\newtheorem{cor}[subsection]{Corollary}
\newtheorem{lma}[subsection]{Lemma}

\theoremstyle{remark}
\newtheorem{rmk}[subsection]{Remark}

\newtheorem{exm}[subsection]{Example}

\def\PP{\mathcal{P}}
\def\RR{\mathcal{R}}
\def\QQ{\mathcal{Q}}
\def\EE{\mathcal{E}}
\def\FF{\mathcal{F}}
\def\DD{\mathcal{D}}

\def\AA{\mathcal{A}}
\def\KK{\mathcal{K}}
\def\II{\underline{I}}
\def\uu{\underline{u}}
\def\vv{\underline{v}}
\def\Mod{\mathrm{Mod}}

\def\Pairs{\mathrm{Pairs}}
\def\Oper{\mathrm{Oper}}
\def\Env{\mathrm{Env}}
\def\Hom{\underline{\mathrm{Hom}}}
\def\End{\underline{\mathrm{End}}}
\def\Aut{\mathrm{Aut}}
\def\Alg{\mathrm{Alg}}
\def\Ho{\mathrm{Ho}}
\def\Top{\mathrm{Top}}
\def\Ass{\mathrm{Ass}}
\def\Comon{\mathrm{Comon}}
\def\lrto{\longrightarrow}
\def\dto{\rightrightarrows}
\def\lra{\leftrightarrows}
\def\eqv{\overset\sim\lrto}
\def\ito{\rightarrowtail}
\def\Sg{\Sigma}
\def\al{\alpha}
\def\be{\beta}
\def\eps{\epsilon}

\begin{document}
\title{On the derived category of an algebra over an operad}

\author{Clemens Berger and Ieke Moerdijk}

\date{16 January 2008}

\begin{abstract}We present a general construction of the derived category of an algebra over an operad and establish its invariance properties. A central role is played by the enveloping operad of an algebra over an operad.\end{abstract}

\maketitle

\section*{Introduction}

It is a classical device in homological algebra to associate to an associative ring $R$ the homotopy category of differential graded $R$-modules, the so-called \emph{derived category} $\DD(R)$ of $R$. One of the important issues is to know when two rings have equivalent derived categories; positive answers to this question may be obtained by means of the theory of tilting complexes, which is a kind of derived Morita theory, cf. Rickard \cite{R}, Keller \cite{K}, Schwede \cite{S}, To\"en \cite{T}. In this paper, we provide a solution to the problem of giving a suitable construction of the derived category associated to an algebra over an operad in a non-additive context. Building on earlier work of ours' (cf. \cite{BM,BM2} and the Appendix to this paper) we establish general invariance properties of this derived category under change of algebra, change of operad and change of ambient category. In the special case of the operad for differential graded algebras, our construction agrees with the classical one. 

A central role in our proofs is played by the \emph{enveloping operad} $\PP_A$ of $A$ whose algebras are the $\PP$-algebras \emph{under} $A$. Indeed, the monoid of unary operations $\PP_A(1)$ may be identified with the \emph{enveloping algebra} $\Env_\PP(A)$ of $A$. The latter has the characteristic property that $\Env_\PP(A)$-modules (in the classical sense) correspond to $A$-modules (in the operadic sense). This indirect construction of the enveloping algebra occurs in specific cases at several places in the literature (cf. Getzler-Jones \cite{GJ}, Ginzburg-Kapranov \cite{GK}, Fresse \cite{F,F2}, Spitzweck \cite{Sp}, van der Laan \cite{vanL}, Basterra-Mandell \cite{BaM}). The main point in the use of the enveloping operad $\PP_A$ rather than the enveloping algebra $\PP_A(1)$ is that the assignment $(\PP, A)\mapsto\PP_A$ extends to a left adjoint functor which on admissible $\Sg$-cofibrant operads $\PP$ and cofibrant $\PP$-algebras $A$ behaves like a left Quillen functor on cofibrant objects. It is precisely this good homotopical behaviour that allows the definition of the derived category $\DD_\PP(A)$ for \emph{any} $\PP$-algebra $A$.\vspace{1ex}

\emph{Acknowledgements :} We are grateful to B. Jahren, K. Hess and B. Oliver, the organizing committee of the Algebraic Topology Semester 2006 of the Mittag-Leffler-Institut in Stockholm. Most of this work has been carried out during the authors' visit at the MLI in Spring 2006. We are also grateful to B. Fresse for helpful comments on an earlier version of this paper.

\section{Enveloping operads and enveloping algebras}

Let $\EE$ be a (bicomplete) closed symmetric monoidal category. We write $I$ for the unit, $-\otimes-$ for the monoidal structure and $\Hom_\EE(-,-)$ for the internal hom of $\EE$.

This section is a recollection of known results on categories of modules over a $\PP$-algebra $A$, where $\PP$ is any symmetric operad in $\EE$. The main objective of this section is to fix the notations and definitions we use. The category of $A$-modules is a module category in the classical sense for a suitable monoid in $\EE$, the so-called \emph{enveloping algebra} of $A$. We will explain in detail that the enveloping algebra of $A$ is isomorphic to the monoid of unary operations $\PP_A(1)$ of the so-called \emph{enveloping operad} of the pair $(\PP,A)$. The enveloping operad $\PP_A$ is characterised by a universal property which implies in particular that $\PP_A$-algebras are the $\PP$-algebras \emph{under} $A$.

\begin{dfn}\label{maindefinition}Let $A$ be a $\PP$-algebra in $\EE$. An $A$-module (under $\PP$) consists of an object $M$ of $\EE$ together with action maps $$\mu_{n,k}:\PP(n)\otimes A^{\otimes k-1}\otimes M\otimes A^{\otimes n-k}\to M,\quad 1\leq k\leq n,$$subject to the following three axioms:\begin{enumerate}\item (Unit axiom) The operad-unit $I\to\PP(1)$ induces a commutative triangle \begin{diagram}[silent,small,UO]I\otimes M&\rTo^\cong&M\\\dTo&\ruTo_{\mu_{1,1}}&\\\PP(1)\otimes M&&\end{diagram}\item (Associativity axiom) For each $n=n_1+\cdots+n_s\geq 1$, the following diagram commutes:\begin{diagram}[silent,small,UO]\PP(s)\otimes\PP(n_1)\otimes\cdots\otimes\PP(n_s)\otimes A^{\otimes k-1}\otimes M\otimes A^{\otimes n-k}&\rTo^\be&\PP(s)\otimes A^{\otimes l-1}\otimes M\otimes A^{\otimes s-l}\\\dTo^\al&&\dTo_{\mu_{s,l}}\\\PP(n)\otimes A^{\otimes k-1}\otimes M\otimes A^{\otimes n-k}&\rTo_{\mu_{n,k}}&M\end{diagram}where $\al$ is induced by the operad structure of $\PP$, and $\be$ is induced by the $\PP$-algebra structure of $A$ and the $A$-module structure of $M$; in particular, $l$ is the unique natural number such that $n_1+\cdots+n_{l-1}< k\leq n_1+\cdots+n_l$.\item (Equivariance axiom) For each $n\geq 1$, the $\mu_{n,k}$ induce a total action \begin{diagram}[silent,small,UO]\mu_n:\PP(n)\otimes_{\Sg_n} (\coprod_{k=1}^{k=n}A^{\otimes k-1}\otimes M\otimes A^{\otimes n-k}) &\rTo& M\end{diagram}where the symmetric group $\Sg_n$ acts on the coproduct by permuting factors.\vspace{1ex}

A morphism $f:M\to N$ of $A$-modules under $\PP$ is a morphism in $\EE$ rendering commutative all diagrams of the form\begin{diagram}[silent,small,UO]\PP(n)\otimes A^{\otimes k-1}\otimes M\otimes A^{\otimes n-k}&\rTo^{id\otimes id\otimes f\otimes id}&\PP(n)\otimes A^{\otimes k-1}\otimes N\otimes A^{\otimes n-k}\\\dTo^{\mu^M_{n,k}}&&\dTo_{\mu^N_{n,k}}\\M&\rTo_f&N.\end{diagram}\end{enumerate}\end{dfn}

The category of $A$-modules under $\PP$ will be denoted by $\Mod_\PP(A)$; the forgetful functor $(M,\mu)\mapsto M$ will be denoted by $U_A:\Mod_\PP(A)\to\EE$.\vspace{1ex}

\begin{rmk}\label{linear}

The pairs $(A,M)$ consisting of a $\PP$-algebra $A$ and an $A$-module $M$ define a category with morphisms the pairs $(\phi,\psi):(A,M)\to(B,N)$ consisting of a map of $\PP$-algebras $\phi:A\to B$ and a map of $A$-modules $\psi:M\to\phi^*(N)$. This category can be identified with the full subcategory of \emph{left $\PP$-modules} concentrated in degrees $0$ and $1$. In order to make this more explicit, recall that a left $\PP$-module $M$ consists of a collection $(M_k)_{k\geq 0}$ of $\Sg_k$-objects together with a map of collections $\PP\circ M\to M$ satisfying the usual axioms of a left action. Such a collection $(M_k)_{k\geq 0}$ is concentrated in degrees $0$ and $1$ precisely when all $M_k$ for $k\geq 2$ are initial objects in $\EE$. The left $\PP$-module structure restricted to $M_0$ endows $M_0$ with a $\PP$-algebra structure, while the left $\PP$-module structure restricted to $(M_0,M_1)$ amounts precisely to an $M_0$-module structure on $M_1$ under $\PP$.

There is yet another way to specify such a pair $(A,M)$ if $\EE$ is an \emph{additive} category. Recall that a $\PP$-algebra structure on the object $A$ of $\EE$ is equivalent to an operad map $\PP\to\End_A$ taking values in the \emph{endomorphism operad} of $A$. The latter is defined by $$\End_A(n)=\Hom_\EE(A^{\otimes n},A)$$where the operad structure maps are given by substitution and permutation of the factors in the domain. For a pair of objects $A$ and $M$ in an additive category $\EE$, we define a \emph{linear endomorphism operad} $\End_{M|A}$ of $M$ relative to $A$ such that operad maps $\PP\to\End_{M|A}$ correspond to a $\PP$-algebra structure on $A$ together with an $A$-module structure on $M$. 

This linear endomorphism operad $\End_{M|A}$ is defined as a \emph{suboperad} of the endomorphism operad $\End_{M\oplus A}$, where $-\oplus-$ stands for the direct sum in $\EE$. Since $$\Hom_\EE(X\oplus Y,Z\oplus W)\cong\begin{pmatrix} \Hom_\EE(X,Z)&\Hom_\EE(Y,Z)\\\Hom_\EE(X,W)&\Hom_\EE(Y,W)\end{pmatrix}$$with the usual matrix rule for composition, it makes sense to define $\End_{M|A}(n)$ as that subobject of $\End_{M\oplus A}(n)$ that takes the summand $A^{\otimes n}$ to $A$, the summands of the form $A^{\otimes k-1}\otimes M\otimes A^{\otimes n-k}$ to $M$, and all other summands to a null (i.e. initial and terminal) object of $\EE$. It is then readily verified that this subcollection $(\End_{M|A}(n))_{n\geq 0}$ of $(\End_{M\oplus A})(n))_{n\geq 0}$ defines a suboperad $\End_{M|A}$ of $\End_{M\oplus A}$, and that an operad map $\PP\to\End_{M|A}$ determines, and is determined by, a $\PP$-algebra structure on $A$ together with an $A$-module structure on $M$.\end{rmk}

It follows from the preceding considerations that for each $A$-module $M$ in an additive category $\EE$, the direct sum $M\oplus A$ carries a canonical $\PP$-algebra structure, induced by the composite operad map $\PP\to\End_{M|A}\to\End_{M\oplus A}$; the resulting $\PP$-algebra is often denoted by $M\rtimes A$, cf. \cite{GH}. Projection on the second factor defines a map of $\PP$-algebras $M\rtimes A\to A$, hence an object of the category $\Alg_\PP/A$ of $\PP$-algebras over $A$. This assignment extends to a functor $\rho:\Mod_A\to\Alg_\PP/A$. The following lemma is due to Quillen \cite{Q2}; it is the starting point of the definition of the \emph{cotangent complex} of the $\PP$-algebra $A$. 

\begin{lma}Let $A$ be an algebra over an operad $\,\PP$ in an additive, closed symmetric monoidal category $\EE$. Under the above construction, the category of $A$-modules is isomorphic to the category of abelian group objects of $\,\Alg_\PP/A$.\end{lma}

\begin{proof}Since $\EE$ is additive, the category of $A$-modules is additive and any $A$-module carries a canonical abelian group structure in $\Mod_A$. By inspection, the functor $\rho:\Mod_A\to\Alg_\PP/A$ preserves finite products, thus abelian group objects, so that for any $A$-module $M$, the image $\rho(M)$ carries a canonical abelian group structure in $\Alg_\PP/A$. The zero element of this abelian group structure is given by the section $A\to M\oplus A$; in particular, the functor $\rho$ is full and faithful, provided $\rho$ is considered as taking values in category of abelian group objects of $\Alg_\PP/A$. 

It remains to be shown that any abelian group object of $\Alg_\PP/A$ arises as $\rho(M)$ for a uniquely determined $A$-module $M$. Indeed, an abelian group object $N\to A$ has a section $A\to N$ by the zero element so that $N$ splits canonically as $N=M\oplus A$. The $\PP$-algebra structure on $N=M\oplus A$ restricts to the given $\PP$-algebra structure on $A$. The abelian group structure $(\al,id_A):(M\oplus M)\oplus A=N\times_AN\lrto N=M\oplus A$ commutes with the $\PP$-algebra structure of $N$; thus, the square\begin{diagram}[small]\PP(2)\otimes (M\oplus M)\otimes (M\oplus M)&\rTo^{id_{\PP(2)}\otimes\al\otimes\al}&\PP(2)\otimes M\otimes M\\\dTo^{\mu_2\oplus\mu_2}&&\dTo_{\mu_2}\\ M\oplus M&\rTo_\al&M\end{diagram}is commutative which implies that $\mu_2$ is zero. This shows that the operad action $\PP\to\End_{M\oplus A}$ factors through $\End_{M|A}$ and we are done.\end{proof}

\begin{lma}\label{Lemma0}Let $\PP$ be an operad. The category of $\PP(0)$-modules under $\PP$ is canonically isomorphic to the module category of the monoid $\PP(1)$.\end{lma}

\begin{proof}For a $\PP(0)$-module $M$ under $\PP$, the action map $\mu_{1,1}:\PP(1)\otimes M\to M$ defines an action on $M$ by the monoid $\PP(1)$. Conversely, an action on $M$ by $\PP(1)$ extends uniquely to action maps $\mu_{n,k}:\PP(n)\otimes\PP(0)^{\otimes k-1}\otimes M\otimes\PP(0)^{\otimes n-k}\to M$ where we use the symmetry of the monoidal structure as well as the operad structure maps $$\PP(n)\otimes\PP(0)^{\otimes k-1}\otimes I\otimes\PP(0)^{\otimes n-k}\to\PP(n)\otimes\PP(0)^{\otimes k-1}\otimes\PP(1)\otimes\PP(0)^{\otimes n-k}\to\PP(1).$$\end{proof}

\begin{dfn}Let $\PP$ be an operad and $A$ be a $\PP$-algebra. 

The \emph{enveloping operad} $\PP_A$ of the $\PP$-algebra $A$ is defined by the universal property that operad maps $\PP_A\to \QQ$ correspond precisely to pairs $(\phi,\psi)$ consisting of an operad map $\phi:\PP\to\QQ$ and a $\PP$-algebra map $\psi:A\to\phi^*\QQ(0)$, and that this correspondence is natural in $\QQ$.\end{dfn}

Alternatively, we can consider the category $\Pairs(\EE)$ of pairs $(\PP,A)$ consisting of an operad $\PP$ and a $\PP$-algebra $A$, with morphisms the pairs $(\phi,\psi):(\PP,A)\to(\QQ,B)$ consisting of an operad map $\phi:\PP\to\QQ$ and a $\PP$-algebra map $\psi:A\to\phi^*(B)$. There is a canonical embedding of the category $\Oper(\EE)$ of operads in $\EE$ into the category $\Pairs(\EE)$ given by $\PP\mapsto(\PP,\PP(0))$. The universal property of the enveloping operad then expresses (provided it exists for all $\PP$ and $A$) that $\Oper(\EE)$ is a \emph{reflective subcategory} of $\Pairs(\EE)$ and that the left adjoint of the embedding is precisely the enveloping operad construction $\Pairs(\EE)\to\Oper(\EE):(\PP,A)\mapsto\PP_A$. In particular, if this left adjoint exists, it preserves all colimits.

\begin{prp}The enveloping operad $\,\PP_A$ exists for any $\PP$-algebra $A$.\end{prp}

\begin{proof}For a free $\PP$-algebra $A={\FF_\PP}(X)$, where $X$ is an object of $\EE$, the enveloping operad of ${\FF_\PP}(X)$ is given by $$\PP_{{\FF_\PP}(X)}(n)=\coprod_{k\geq 0}\PP(n+k)\otimes_{\Sg_k}X^{\otimes k},$$see for instance Getzler and Jones \cite{GJ}. A general $\PP$-algebra $A$ is part of a canonical coequalizer$${\FF_\PP}{\FF_\PP}(A)\dto{\FF_\PP}(A)\to A,$$whence the corresponding coequalizer of operads\begin{gather}\label{coeq1}\PP_{{\FF_\PP}{\FF_\PP}(A)}\dto\PP_{{\FF_\PP}(A)}\to\PP_A\end{gather}has the required universal property of the enveloping operad of $A$.\end{proof}

The identity $\PP_A\to\PP_A$ corresponds by the universal property to an operad map $\eta_A:\PP\to\PP_A$ together with a map of $\PP$-algebras $\bar{\eta}_A:A\to\eta_A^*\PP_A(0)$. We will now show that the latter map is an isomorphism.

\begin{lma}\label{Lemma1}For any $\PP$-algebra $A$, the category of $\PP_A$-algebras is canonically isomorphic to the category of $\PP$-algebras under $A$, and $\PP_A(0)$ is isomorphic to $A$.\end{lma}

\begin{proof}The pair $(\eta_A,\bar{\eta}_A)$ induces a functor from the category of $\PP_A$-algebras to the category of $\PP$-algebras under $A$ which is compatible with the forgetful functors. This functor is an \emph{isomorphism} of categories since a $\PP_A$-algebra structure on $B$ is given equivalently by an operad map $\PP_A\to\End_B$ or by an operad map $\PP\to\End_B$ (i.e. a $\PP$-algebra structure on $B$) together with a map of $\PP$-algebras $A\to B$.\end{proof}

\begin{lma}\label{Lemma2}Let $\PP$ be an operad and $\al:A\to B$ be a map of $\PP$-algebras. Write $B_\al$ for the $\PP_A$-algebra defined by $\al$. The enveloping operad of the $\PP$-algebra $B$ is isomorphic to the enveloping operad of the $\PP_A$-algebra $B_\al$.\end{lma}

\begin{proof}An operad map $(\PP_A)_{B_\al}\to Q$ gives rise to a pair $(\phi,\psi)$ consisting of an operad map $\phi:\PP_A\to Q$ and a $\PP_A$-algebra map $\psi:B_\al\to\phi^*Q(0)$. According to Lemma \ref{Lemma1}, the latter yields a $\PP$-algebra map $\psi':B\to\eta_A^*\phi^*Q(0)$ (under $A$) for the canonical operad map $\eta_A:\PP\to\PP_A$. Conversely, the pair $(\phi\eta_A,\psi')$ uniquely determines the operad map $(\PP_A)_{B_\al}\to Q$ we started from. Therefore, the enveloping operad $(\PP_A)_{B_\al}$ has the same universal property as the enveloping operad $\PP_B$ so that both operads are isomorphic.\end{proof}

The following proposition is preparatory for the relationship between enveloping operad and enveloping algebra. The result is implicitly used by Goerss and Hopkins, compare \cite[Lemma 1.13]{GH}.


\begin{prp}\label{cotensor}Let $T$ be a monad on a closed symmetric monoidal category $\EE$. 

The category of $T$-algebras is a module category for a monoid $M_T$ in $\EE$ if and only if the tensor-cotensor adjunction of $\EE$ lifts to the category of $T$-algebras along the forgetful functor $U_T:\Alg_T\to\EE$. If the latter is the case, the monad $T$ is isomorphic to $T(I)\otimes(-)$, i.e. the monoid $M_T$ is given by $T(I)$.\end{prp}

\begin{proof}Assume first that the monad $T$ is given by tensoring with a monoid $M_T$. Then the tensor-cotensor adjunction of $\EE$ lifts as follows. For any $M_T$-modules $M,N$ and object $X$ of $\EE$, the tensor $M\otimes X$ inherits an $M_T$-module structure by \begin{diagram}[small]M_T\otimes M\otimes X&\rTo^{\eps_M\otimes X}& M\otimes X;\end{diagram}the cotensor $\Hom_\EE(X,N)$ inherits an $M_T$-module structure by the adjoint of \begin{diagram}[small]M_T\otimes\Hom_\EE(X,N)\otimes X&\rTo^{M_T\otimes ev_X} &M_T\otimes N&\rTo^{\eps_N}& N.\end{diagram}It follows that the adjunction $\EE(M\otimes X,N)\cong\EE(M,\Hom_\EE(X,N))$ lifts to an adjunction $\Alg_T(M\otimes X,N)\cong\Alg_T(M,\Hom_\EE(X,N))$.

Assume conversely that such a lifted tensor-cotensor adjunction exists for a given monad $T$ on $\EE$. By adjointness we get for any objects $X,Y$ binatural isomorphims of $T$-algebras$$F_T(X)\otimes Y\cong F_T(X\otimes Y).$$Since by assumption $U_T$ preserves tensors this implies (setting $X=I$) that the monad $T=U_TF_T$ is isomorphic to $T(I)\otimes(-)$;  in particular, $T(I)$ carries a canonical monoid structure.\end{proof}

\begin{thm}\label{envelope}For any algebra $A$ over an operad $\PP$ in a closed symmetric monoidal category $\EE$, the category of $A$-modules under $\PP$ is canonically isomorphic to the module category of the monoid $\PP_A(1)$.\end{thm}

\begin{proof}First of all, it follows immediately from Definition \ref{maindefinition} that the forgetful functor $U_A:\Mod_A\to\EE$ creates colimits (hence permits an application of Beck's tripleability theorem) and allows a lifting to $\Mod_A$ of the tensors and cotensors by objects of $\EE$. In order to apply Proposition \ref{cotensor}, and to compare the resulting monoid to $\PP_A(1)$, we give an explicit description of the left adjoint $\FF_A:\EE\to\Mod_A$ of $U_A$, again following Goerss and Hopkins, compare \cite[Proposition 1.14]{GH}. For objects $M$ and $A$ of $\EE$, denote by $\Psi(A,M)$ the positive collection in $\EE$ given by$$\Psi(A,M)(n)=\coprod_{k=1}^nA^{\otimes k-1}\otimes M\otimes A^{\otimes n-k},\quad n\geq 1,$$the symmetric group $\Sg_n$ acting by permutation of the factors. Moreover, define the object $\PP(A,M)=\coprod_{n\geq 1}\PP(n)\otimes_{\Sg_n}\Psi(A,M).$ The axioms of an $A$-module $M$ amount then to the existence of an action map $\mu_M:\PP(A,M)\to M$ which is unitary and associative in a natural sense. For instance, the associativity constraint uses a canonical isomorphism $(\PP\circ\PP)(A,M)\cong\PP({\FF_\PP}(A),\PP(A,M))$ where ${\FF_\PP}(A)$ is the free $\PP$-algebra on $A$. It follows that for a free $\PP$-algebra $A={\FF_\PP}(X)$ the free $A$-module on $M$ is given by $\PP(X,M)$, the $A$-module structure being induced by the isomorphim just cited. A general $\PP$-algebra $A$ is part of a reflexive coequalizer$${\FF_\PP}{\FF_\PP}(A)\dto{\FF_\PP}(A)\to A$$which is preserved under the forgetful functor $\Alg_\PP\to\EE$. Therefore, the underlying object of the free $A$-module $\FF_A(M)$ on $M$ is part of a reflexive coequalizer in $\EE$ \begin{gather}\label{coeq2}\PP({\FF_\PP}(A),M)\dto\PP(A,M)\to U_A\FF_A(M).\end{gather}Proposition \ref{cotensor} implies that the category $\Mod_A$ is a module category for the monoid $M_A\cong U_A\FF_A(I)$. Putting $M=I$ in (\ref{coeq2}) we end up with the following reflexive coequalizer diagram in $\EE$\begin{gather}\label{coeq3}\PP({\FF_\PP}(A),I)\dto\PP(A,I)\to M_A.\end{gather}

For the second step of the proof observe first that the coequalizer (\ref{coeq1}) in $\Oper(\EE)$ is preserved under the forgetful functor from operads to collections, since operads are monoids in collections with respect to the circle product, and since the coequalizer is reflexive. Therefore we get the following reflexive coequalizer diagram in $\EE$\begin{gather}\label{coeq4}\PP_{{\FF_\PP}{\FF_\PP}(A)}(1)\dto\PP_{{\FF_\PP}(A)}(1)\to\PP_A(1).\end{gather}

It follows from the definitions that (\ref{coeq3}) and (\ref{coeq4}) are isomorphic diagrams in $\EE$. It remains to be shown that the monoid structures of $M_A$ and of $\PP_A(1)$ coincide under this isomorphism.  Lemmas \ref{Lemma0} and \ref{Lemma1} imply that the category of $\PP_A(1)$-modules is isomorphic to the category of $A$-modules under $\PP_A$; the canonical operad map $\eta_A:\PP\to\PP_A$ induces thus a functor (over $\EE$) from the category of $\PP_A(1)$-modules to the category of $A$-modules under $\PP$, and therefore (by Proposition \ref{cotensor}) a map of monoids from $\PP_A(1)$ to $M_A$; this map of monoids may be identified with the isomorphism between the coequalizers of (\ref{coeq4}) and (\ref{coeq3}).\end{proof}

\begin{dfn}\label{envalgebra}For any algebra $A$ over an operad $\PP$ in a closed symmetric monoidal category $\EE$, the \emph{enveloping algebra} $\Env_\PP(A)$ is the monoid $\PP_A(1)$ of unary operations of the enveloping operad of $A$.\end{dfn}

The enveloping algebra construction is a functor that takes maps of pairs $(\phi,\psi):(\PP,A)\to(\QQ,B)$ to maps of monoids $\Env_\PP(A)\to\Env_\QQ(B)$ in $\EE$. Theorem \ref{envelope} shows that the category of $A$-modules under $\PP$ is canonically isomorphic to the module category of the enveloping algebra $\Env_\PP(A)$.

The purpose of the remaining part of this section is to give a sufficient condition for the enveloping algebra $\Env_\PP(A)$ to be a \emph{bialgebra}, i.e. to have a compatible comonoid structure; this amounts to the existence of a monoidal structure on the category of $A$-modules under $\PP$.

Recall that a \emph{Hopf operad} $\PP$ in $\EE$ is by definition an operad in the symmetric monoidal category $\Comon(\EE)$ of comonoids in $\EE$; for such a Hopf operad, a \emph{$\PP$-bialgebra} is defined to be a $\PP$-algebra in $\Comon(\EE)$. Alternatively, a ``Hopf structure'' on an operad $\PP$ amounts to a \emph{monoidal structure} on the category of $\PP$-algebras such that the forgetful functor is strongly monoidal, cf. \cite{Moe}; $\PP$-bialgebras are then precisely comonoids in this monoidal category of $\PP$-algebras.

For any two operads $\PP$ and $\QQ$, the tensor product $\PP\otimes\QQ$ denotes the operad defined by $(\PP\otimes\QQ)(n)=\PP(n)\otimes\QQ(n)$. 

\begin{prp}For any Hopf operad $\PP$ and $\PP$-bialgebra $A$, the enveloping operad $\,\PP_A$ is again a Hopf operad. In particular, the enveloping algebra of $A$ is a bialgebra in $\EE$.\end{prp}

\begin{proof}Any Hopf operad $\PP$ has a diagonal $\PP\to\PP\otimes\PP$. Therefore, for any $\PP$-algebras $A$ and $B$, there is a canonical operad map $\PP\to\PP\otimes\PP\overset{\eta_A\otimes\eta_B}{\lrto}\PP_A\otimes\PP_B$, and hence by the universal property of $\PP_{A\otimes B}$, a canonical operad map $\PP_{A\otimes B}\to\PP_A\otimes\PP_B$. 

If $A$ is a $\PP$-bialgebra, there is a diagonal $A\to A\otimes A$ in the category of $\PP$-algebras and hence a diagonal $\PP_A\to\PP_{A\otimes A}\to\PP_A\otimes\PP_A$ in the category of operads. This shows that $\PP_A$ is a Hopf operad, and that $\PP_A(1)$ is a bialgebra in $\EE$.\end{proof}

\section{The derived category of an algebra over an operad}

In this section, we study the derived category of an algebra over an operad. For any $\PP$-algebra $A$, the derived category $\DD_\PP(A)$ is defined to be the homotopy category of the category of $cA$-modules, where $cA$ is a cofibrant resolution of $A$ in the category of $\PP$-algebras. Thanks to Theorem \ref{envelope}, invariance properties of $\DD_\PP(A)$ correspond to invariance properties of the enveloping algebra $\Env_\PP(A)$. Since the enveloping algebra may be identified with the monoid of unary operations of the enveloping operad $\PP_A$, the methods of \cite{BM} apply (see the Appendix for a small correction to \cite{BM}), and we get quite precise information on the invariance properties of the derived category $\DD_\PP(A)$. 

From now on, $\EE$ denotes a \emph{monoidal model category}. Recall (cf. Hovey \cite{Hov}) that a monoidal model category is simultaneously a closed symmetric monoidal category and a Quillen model category such that two compatibility axioms hold: the \emph{pushout-product axiom} and the \emph{unit axiom}. The unit axiom requires the existence of a cofibrant resolution of the unit $cI\eqv I$ such that tensoring with cofibrant objects $X$ induces weak equivalences $X\otimes cI\eqv X$. The latter is of course automatic if the unit of $\EE$ is already cofibrant. We assume throughout that $\EE$ is \emph{cocomplete} and \emph{cofibrantly generated} as a model category.

For any unitary associative ring $R$, the category of \emph{simplicial $R$-modules} is an additive monoidal model category with weak equivalences (resp. fibrations) those maps of simplicial $R$-modules whose underlying map is a weak equivalence (resp. fibration) of simplicial sets. Similarly, the category of \emph{differential graded $R$-modules} is an additive monoidal model category with weak equivalences (resp. fibrations) the quasi-isomorphisms (resp. epimorphisms) of differential graded $R$-modules. These two examples generalise to any \emph{Grothendieck abelian category} $\AA$ equipped with a set of generators.

Recall from \cite{BM2} that a map of operads is called a \emph{$\Sg$-cofibration} if the underlying map is a cofibration of collections and an operad $\PP$ is called \emph{$\Sg$-cofibrant} if the unique map from the initial operad to $\PP$ is a $\Sg$-cofibration. In particular, for a $\Sg$-cofibrant operad $\PP$, the unit $I\to\PP(1)$ is a cofibration in $\EE$. This terminology differs slightly from \cite{BM} where an operad with the latter property has been called well-pointed.

An operad $\PP$ is called \emph{admissible} if the model structure on $\EE$ transfers to the category $\Alg_\PP$ of $\PP$-algebras along the free-forgetful adjunction $$\FF_\PP:\EE\lra\Alg_\PP:U_\PP,$$ i.e. if $\Alg_\PP$ carries a model structure whose weak equivalences (resp. fibrations) are those maps $f:A\to B$ of $\PP$-algebras for which the underlying map $U_\PP(f):U_\PP(A)\to U_\PP(B)$ is  weak equivalence (resp. fibration) in $\EE$. See \cite[Proposition 4.1]{BM} for conditions on $\PP$ which imply admissibility.

\begin{lma}\label{basic0}For any admissible operad $\PP$ and $\PP$-algebra $A$, the enveloping operad $\PP_A$ is again admissible.\end{lma}

\begin{proof}This follows immediately from Lemma \ref{Lemma1}.\end{proof}

A (trivial) \emph{cellular extension} of $\PP$-algebras is any sequential colimit of pushouts $A\to A[u]$ of the form\begin{diagram}[UO,small,silent]\FF_\PP U_\PP(A)&\rTo^{\eps_A}&A\\\dTo^{\FF_\PP(u)}&&\dTo\\\FF_\PP(X)&\rTo &A[u]\end{diagram}where $\eps_A$ denotes the counit of the free-forgetful adjunction and $u:U_\PP(A)\to X$ is a generating (trivial) cofibration in $\EE$. A \emph{cellular $\PP$-algebra} $A$ is a cellular extension of the initial $\PP$-algebra $\PP(0)$.

\begin{lma}\label{basic}Let $\PP$ be a $\Sg$-cofibrant operad and $A$ be a $\PP$-algebra. If the unique map of $\PP$-algebras $\PP(0)\to A$ is a (trivial) cellular extension of $\PP$-algebras, then the induced map $\PP\to\PP_A$ is a (trivial) $\Sg$-cofibration of operads.\end{lma}

\begin{proof}The case of a cellular extension is \cite[Proposition 5.4]{BM}. Exactly the same proof applies to a trivial cellular extension as well.\end{proof}

A monoid $M$ in $\EE$ will be called \emph{well-pointed} if the unit map $I\to M$ is a cofibration in $\EE$. In particular, if $I$ is cofibrant then a well-pointed monoid $M$ has a cofibrant underlying object. 

\begin{prp}\label{cofibrantextension}Let $\PP$ be an admissible $\Sg$-cofibrant operad and $A$ be a cofibrant $\PP$-algebra. Then the enveloping operad $\PP_A$ is an admissible $\Sg$-cofibrant operad. In particular, the enveloping algebra $\Env_\PP(A)$ is well-pointed.\end{prp}

\begin{proof}Admissibility was dealt with in Lemma \ref{basic0}. Any cofibrant $\PP$-algebra is retract of a cellular $\PP$-algebra, whence $\PP_A$ is retract of $\Sg$-cofibrant operad and therefore $\Sg$-cofibrant. The second statement follows from the identification of the enveloping algebra $\Env_\PP(A)$ with $\PP_A(1)$.\end{proof}

\begin{cor}\label{invariance}Let $\PP$ be an  admissible $\Sg$-cofibrant operad and $\al:A\to B$ be a weak equivalence of cofibrant $\PP$-algebras. The induced map $\PP_\al:\PP_A\to\PP_B$ is a weak equivalence of admissible $\Sg$-cofibrant operads. In particular, the induced map of enveloping algebras $\Env_\PP(A)\to\Env_\PP(B)$ is a weak equivalence of well-pointed monoids.\end{cor}

\begin{proof}By K. Brown's Lemma, it suffices to consider the case of a trivial cellular extension $\al$. The statement then follows from Lemmas \ref{Lemma2}, \ref{basic0} and \ref{basic}, since $\PP_B$ may be identified with $(\PP_A)_{B_\al}$, and $\PP_A$ is an admissible $\Sg$-cofibrant operad.\end{proof}

\begin{rmk}The homotopical properties of the enveloping operad construction $(\PP,A)\mapsto \PP_A$, as expressed by Lemma \ref{basic} and Corollary \ref{invariance}, are the main technical ingredients in establishing the homotopy invariance of the derived category of an algebra over an operad. Beno\^\i t Fresse pointed out to us that in recent work \cite{F2}, he independently obtained similar homotopical properties of the assignment $(\RR,A)\mapsto\RR\circ_{\PP}A$ where $\RR$ is a $\Sg$-cofibrant right $\PP$-module and $A$ is a cofibrant $\PP$-algebra. Since the enveloping operad $\PP_A$ may be identified with $\RR\circ_\PP A$ for a certain right $\PP$-module $\RR$, see \cite[Section 10]{F2}, Lemma \ref{basic} and Corollary \ref{invariance} may be recovered from \cite[Lemma 13.1.B]{F2} and \cite[Theorem 13.A.2]{F2}. The more general context of right $\PP$-modules however makes the proofs of these statements more involved than those of our results which are immediate consequences of \cite[Section 5]{BM}.\end{rmk}

\begin{thm}\label{changeofalgebra}Let $\PP$ be an admissible $\Sg$-cofibrant operad in a cofibrantly generated monoidal model category. For any cofibrant $\PP$-algebra $A$, the category $\Mod_\PP(A)$ of $A$-modules carries a transferred model structure. 

Any map of cofibrant $\PP$-algebras $f:A\to B$ induces a Quillen adjunction $f_!:\Mod_\PP(A)\to\Mod_\PP(B):f^*$. If $f$ is a weak equivalence, then $(f_!,f^*)$ is a Quillen equivalence. \end{thm}

\begin{proof}By Proposition \ref{cofibrantextension}, the enveloping algebra of a cofibrant $\PP$-algebra is well-pointed. Theorem \ref{envelope}, Corollary \ref{invariance} and Proposition \ref{changeofmonoid} then yield the conclusion.\end{proof}

\begin{prp}\label{changeofmonoid}Let $\EE$ be a cofibrantly generated monoidal model category.
\begin{itemize}\item[(a)]For any well-pointed monoid $M$, the category $\Mod_M$ of $M$-modules carries a transferred model structure; if in addition $M$ has a compatible cocommutative comonoid structure, then $\Mod_M$ is a monoidal model category;\item[(b)]Each map of well-pointed monoids $f:M\to N$ induces a Quillen adjunction $f_!:\Mod_M\lra\Mod_N:f^*$; if $f$ is a weak equivalence, then $(f_!,f^*)$ is a Quillen equivalence.\end{itemize}
\end{prp}

\begin{proof}The first part of (a) follows by a transfer argument (cf. \cite[Section 2.5]{BM}) from the fact that tensoring with $M$ preserves colimits, as well as cofibrations and trivial cofibrations; the preservation of cofibrations and trivial cofibrations follows from the pushout-product axiom and the well-pointedness of $M$. If in addition $M$ has a compatible comonoid structure, then $\Mod_M$ carries a closed monoidal structure which is strictly preserved by the forgetful functor $\Mod_M\to\EE$. Since the forgetful functor preserves and reflects limits as well as fibrations and weak equivalences, this implies that $\Mod_M$ satisfies the axioms of a monoidal model category.

The first statement of (b) follows, since $f^*$ preserves fibrations and trivial fibrations by definition of the transferred model structures on $\Mod_M$ resp. $\Mod_N$. For the second statement of (b), we use that $(f_!,f^*)$ is a Quillen equivalence if and only if the unit $\eta_X:X\to f^*f_!(X)$ is a weak equivalence at each cofibrant $M$-module $X$. Assume that $f$ is a weak equivalence; since any cofibrant $M$-module is retract of a ``cellular extension'' of the initial $M$-module, Reedy's patching and telescope lemmas (cf. \cite[Section 2.3]{BM2}) imply that it is sufficient to consider $M$-modules of the form  $M\otimes C$, where $C$ is a cofibrant object of $\EE$. In this case, the unit $\eta_{M\otimes C}$ may be identified with $f\otimes id_C:M\otimes C\to N\otimes C$; the latter is a weak equivalence by an application of the pushout-product axiom and of K. Brown's Lemma to the functor $(-)\otimes C:I\slash\EE\to C\slash\EE$.\end{proof} 


\begin{rmk}\label{cofibrant}It follows from Proposition \ref{changeofmonoid} (b) that for each weak equivalence $f:M\to N$ of well-pointed monoids, the unit $\eta_X:X\to f^*f_!(X)$ is a weak equivalence at cofibrant $M$-modules $X$. For later use, we observe that this holds also if $X$ is only cofibrant as an object of $\EE$ provided $N$ is cofibrant as an $M$-module.\end{rmk}

\begin{rmk}If $\EE$ satisfies the \emph{monoid axiom} of Schwede and Shipley \cite{SS}, the category of $M$-modules carries a transferred model structure for \emph{any} monoid $M$ in $\EE$. The monoid axiom holds in many interesting situations, in particular if either all objects of $\EE$ are cofibrant, or all objects of $\EE$ are fibrant. However, even if the monoid axiom holds, Proposition \ref{changeofmonoid} (b) does \emph{not} carry over to a base-change along arbitrary monoids; indeed, the unit of the base-change adjunction behaves in general badly at cofibrant $M$-modules if $M$ is not supposed to be well-pointed.\end{rmk}

In general it is more restrictive for $f:A\to B$ to be a weak equivalence than to induce a Quillen equivalence. A complete characterisation of those $f$ which induce a Quillen equivalence on module categories would require a homotopical Morita theory. We give here, in a particular case, a precise criterion for when a Quillen equivalence between module categories comes from a weak equivalence.

\begin{prp}Under the hypotheses of \ref{changeofalgebra}, let $f:A\to B$ be a map of cofibrant $\PP$-algebras. Assume that either $A$ or the enveloping algebra of $B$ is cofibrant as an $A$-module. Then $f$ is a weak equivalence if and only if $(f_!,f^*)$ is a Quillen equivalence and the induced map of $B$-modules $f_!(A)\to B$ is a weak equivalence.\end{prp}

\begin{proof}The given $f$ can be considered as a map of $A$-modules $A\to f^*(B)$, and as such it factors through the unit of the adjunction: $A\to f^*f_!(A)\to f^*(B)$. It follows from \cite[Corollary 5.5]{BM} that $A$ has a cofibrant underlying object, and it follows from Proposition \ref{cofibrantextension} that the enveloping algebras of $A$ and of $B$ are well-pointed.

Assume first that $f$ is a weak equivalence. Then, by Theorem \ref{changeofalgebra}, $(f_!,f^*)$ is a Quillen equivalence, and by Remark \ref{cofibrant}, either of the hypotheses implies that the unit $A\to f^*f_!(A)$ is a weak equivalence, whence $f_!(A)\to B$ is a weak equivalence.

Conversely, assume that $(f_!,f^*)$ is a Quillen equivalence, and that $f_!(A)\to B$, and hence $f^*f_!(A)\to f^*(B)$, is a weak equivalence; then, by Remark \ref{cofibrant}, either of the hypotheses implies that $A\to f^*f_!(A)$ is a weak equivalence, whence $f$ is a weak equivalence.\end{proof}
 
\begin{dfn}Let $\PP$ be an admissible $\Sg$-cofibrant operad. The \emph{derived category} $\DD_\PP(A)$ of a $\PP$-algebra $A$ is the homotopy category of the category of $cA$-modules under $\PP$,  where $cA$ is any cofibrant resolution of $A$ in the category of $\PP$-algebras.\end{dfn}

Up to adjoint equivalence, this definition does not depend on the choice of the cofibrant resolution, by Theorem \ref{changeofalgebra}. Furthermore, any map $f:A\to B$ of $\PP$-algebras induces an adjunction $f_!:\DD_\PP(A)\lra\DD_\PP(B):f^*$ since the cofibrant resolutions can be chosen functorially. For weak equivalences $f$, this adjunction is an adjoint equivalence.

We shall now discuss the functorial behaviour of the derived category under change of operads $\phi:\PP\to\QQ$. Recall that $\phi$ induces an adjunction $$\phi_!:\Alg_\PP\lra\Alg_\QQ:\phi^*$$ which is a Quillen adjunction for admissible operads $\PP,\QQ$. In particular, for any $\PP$-algebra $A$, $\phi$ induces a map of enveloping operads $\PP_A\to\QQ_{\phi_!(A)}$, and hence a Quillen adjunction $\Mod_\PP(A)\lra\Mod_\QQ(\phi_!A)$. The derived adjunction of the Quillen pair $(\phi_!,\phi^*)$ is denoted by $$L\phi_!:\Ho(\Alg_\PP)\lra\Ho(\Alg_\QQ):R\phi^*.$$

\begin{thm}\label{comparison1}Let $\EE$ be a left proper, cofibrantly generated monoidal model category and $\phi:\PP\to\QQ$ be a weak equivalence of admissible $\Sg$-cofibrant operads. Then, 

\begin{itemize}\item[(a)]for a cofibrant $\PP$-algebra $A$, the Quillen adjunction $\Mod_\PP(A)\lra\Mod_\QQ(\phi_!A)$ is a Quillen equivalence;\item[(b)]for an arbitrary $\PP$-algebra $A$, the map $\phi$ induces an equivalence of derived categories $\DD_\PP(A)\simeq\DD_\QQ(L\phi_!(A))$;\item[(c)]for an arbitrary $Q$-algebra $B$, $\phi$ induces an equivalence of derived categories $\DD_\PP(R\phi^*(B))\simeq\DD_\QQ(B)$.\end{itemize}\end{thm}

\begin{proof}Part (a) is a special case of \cite[Theorem 4.4]{BM}. More precisely, since $\phi$ is a weak equivalence, the canonical map of operads $\PP_{A}\to\QQ_{\phi_!(A)}$ is a weak equivalence of admissible $\Sg$-cofibrant operads, cf. the proof of \cite[Proposition 5.7]{BM}. Therefore, $\Env_\PP(A)\to\Env_\QQ(\phi_!(A))$ is a weak equivalence and $\Mod_\PP(A)\lra\Mod_\QQ(\phi_!(A))$ is a Quillen equivalence by Theorem \ref{changeofalgebra}. 

Part (b) follows from part (a), applied to a cofibrant resolution $cA$ of $A$.

For (c), let $B$ be a $\QQ$-algebra, and let $B\eqv fB$ be a fibrant resolution of $B$ in the category of $\QQ$-algebras. Then, $\DD_Q(B)\simeq\DD_Q(fB)$ as asserted earlier. By part (a), the canonical map $\phi_! c\phi^*(fB)\to fB$ (adjoint to $c\phi^*(fB)\eqv\phi^*(fB)$) is a weak equivalence inducing another equivalence of derived categories. Also, part (b) gives an equivalence $\DD_\PP(\phi^*(fB))\simeq\DD_\QQ(\phi_!c\phi^*(fB))$. Putting all these equivalences together, we get as required$$\DD_\PP(R\phi^*(B))=\DD_\PP(\phi^*(fB))\simeq\DD_\QQ(\phi_!c\phi^*(fB))\simeq\DD_\QQ(fB)\simeq\DD_\QQ(B).$$\end{proof}

\begin{exm}Let $\EE=\Top$ be the category of compactly generated spaces endowed with Quillen's model structure, let $\PP=A_\infty$ be any $\Sg$-cofibrant $A_\infty$-operad, and let $X$ be an $A_\infty$-space. Then, $\DD_{A_\infty}(X)$ is equivalent to the homotopy category $\Ho(\Top/BX)$ of spaces over the classifying space $BX$ of $X$. Indeed, by Theorem \ref{comparison1}, $X$ may be rectified to a monoid $M=\phi_!(cX)$ along the canonical map of operads $\phi:A_\infty\to\Ass$ without changing the derived category : $\DD_{A_\infty}(X)\simeq\DD_\Ass(L\phi_!X)$, and $BX$ may be identified with the usual classifying space $BM$. The Borel construction then yields an equivalence $\DD_\Ass(M)\simeq\Ho(\Top/BM)$. In particular, if $X$ is the loop space of a pointed connected space $Y$, endowed with the $A_\infty$-action by the little intervals operad, we get $\DD_{A_\infty}(\Omega Y)\simeq\Ho(\Top/Y)$.\end{exm}

For the last comparison theorem, recall that a \emph{monoidal Quillen adjunction} $\Phi_!:\FF\lra\EE:\Phi^*$ between monoidal model categories is an adjoint pair which is simultaneously a Quillen adjunction (with respect to the model structures) and a monoidal adjunction (with respect to the closed symmetric monoidal structures). The latter means that the left and right adjoints of the adjunction are lax symmetric monoidal functors, and that unit and counit of the adjunction are symmetric monoidal natural transformations, see \cite[III.20]{MMSS} for details. Observe that if the left adjoint of an adjunction between closed symmetric monoidal categories is strong symmetric monoidal, then the right adjoint carries a natural symmetric monoidal structure for which the adjunction is monoidal.

\begin{thm}\label{comparison2}Let $\Phi_!:\FF\lra\EE:\Phi^*$ be a monoidal Quillen adjunction between cofibrantly generated monoidal model categories such that the induced map on units $\Phi_!(I_\FF)\to I_\EE$ is a cofibration. 

Let $\QQ$ be a $\Sg$-cofibrant admissible operad in $\FF$, and $B$ be a $\QQ$-algebra, and assume that $\Phi_!Q$ is an admissible operad in $\EE$. Then there is a canonical adjunction of derived categories $\DD_\QQ(B)\lra\DD_{\Phi_!\QQ}(L\Phi_!B)$. This adjunction is a Quillen equivalence, whenever $(\Phi_!,\Phi^*)$ is a Quillen equivalence.\end{thm}

\begin{proof}First notice that the monoidal adjunction induces adjoint functors $\Phi_!:\Oper(\FF)\lra\Oper(\EE):\Phi^*$, and that if $\QQ$ is a $\Sg$-cofibrant operad in $\FF$, then $\Phi_!(Q)$ is so in $\EE$. By assumption, $Q$ and $\Phi_!(Q)$ are both admissible so that we have an induced Quillen adjunction$$\Phi_!:\Alg_\FF(\QQ)\lra\Alg_\EE(\Phi_!(Q)):\Phi^!,$$where $\Phi^!(A)$ is $\Phi^*(A)$ with the $Q$-algebra structure induced by $\eta_Q:Q\to\Phi_!\Phi^*(Q)$.

Now let $cB\eqv B$ be a cofibrant resolution of the $\QQ$-algebra $B$. Then the derived category $\DD_{\Phi_!(\QQ)}(L\Phi_!B)$ is the derived category of modules for the monoid $\Phi_!(Q)_{\Phi_!(cB)}(1)$, while $\DD_Q(B)$ is that for the monoid $\QQ_{cB}(1)$. By Lemma \ref{changeenvoperad} below and the assumption on units, we have an isomorphism of well-pointed monoids $\Phi_!(Q)_{\Phi_!(cB)}(1)\cong\Phi_!(\QQ_{cB})(1)$ in $\EE$ whence Lemma \ref{end} gives the result.\end{proof}

\begin{lma}\label{changeenvoperad}Under the hypotheses of Theorem \ref{comparison2}, each $Q$-algebra $B$ induces a canonical isomorphism of operads $\Phi_!(Q_B)\cong\Phi_!(Q)_{\Phi_!(B)}.$\end{lma}

\begin{proof}It is sufficient to check that both operads enjoy the same universal property. This in turn follows from the fact that for each $\beta:\Phi_!(\QQ)\to\PP$ with adjoint $\gamma:\QQ\to\Phi^*(\PP)$, the composite map $\Phi^!\circ\beta^*:\Alg_\EE(\PP)\to\Alg_\EE(\Phi_!(\QQ))\to\Alg_\FF(\QQ)$ coincides with the composite map $\gamma^*\circ\Phi^*:\Alg_\EE(\PP)\to\Alg_\FF(\Phi^*(\PP))\to\Alg_\FF(\QQ)$.\end{proof}

\begin{lma}\label{end}Under the hypotheses of Theorem \ref{comparison2}, each well-pointed monoid $M$ in $\FF$ induces a well-pointed monoid $\Phi_!(M)$ in $\EE$, and there is a Quillen adjunction$$\Phi_!:\Mod_\FF(M)\lra\Mod_\EE(\Phi_!M):\Phi^!.$$This adjunction is a Quillen equivalence, provided $\Phi_!:\FF\lra\EE:\Phi^*$ is a Quillen equivalence.\end{lma}

\begin{proof}This follows from the fact that cofibrant modules over well-pointed monoids have cofibrant underlying objects, and that $\Phi^!$ is just $\Phi^*$ on underlying objects.\end{proof}

\section{Appendix}

Section 5 of \cite{BM} is used several times in this paper (for instance in the proofs of Lemma \ref{basic} and Theorems \ref{comparison1} and \ref{comparison2}). There is a small mistake in the proof of \cite[Proposition 5.1]{BM} in that construction 5.11 in loc. cit. does not take care of the unit of the operad extension $P[u]$. In order to remedy this, one has to include in the definition of a $\Sg$-cofibrant operad $P$  the condition that the unit $I\to P(1)$ is a cofibration in $\EE$ (as done in Section $2$ of this paper and in \cite{BM2}). Moreover, construction 5.11 has to be slightly adapted in the following way:

First of all, the inductive construction of $F_k(n)$ should be over the set $\AA_k(n)$ of admissible coloured trees with $n$ inputs and $k$ vertices which are \emph{either coloured or unary}. Next, $\uu^-(T,c)\ito\uu(T,c)$ should be replaced by $\uu^*(T,c)\ito\uu(T,c)$ where $\uu^*(T,c)$ consists of those labelled trees which lie in $\uu^-(T,c)$ or have a unary vertex labelled by the identity. In order to show that $\uu^*(T,c)\ito\uu(T,c)$ is again an $\Aut(T,c)$-cofibration (as was done for $\uu^-(T,c)\ito\uu(T,c)$), one has the following variation of \cite[Lemma 5.9]{BM}.

\begin{lma}Let $\II$ be the collection given by the unit $I$ concentrated in degree $1$, and let $\II\ito\KK_1\overset{u}{\ito}\KK_2$ be cofibrations of collections. Then $\uu^*(T,c)\ito\uu(T,c)$ is an $\Aut(T,c)$-cofibration.\end{lma}

\begin{proof}The inductive definition of $\uu(T,c)$ given in \cite[page 828, lines 2-3]{BM} comes together with a similar inductive description of $\uu^*(T,c)$: $$\uu^*(T,c)=\begin{cases}\KK_1(n)\otimes\vv^*(T,c)&\textrm{if the root is uncoloured, not unary},\\\KK_1(1)\otimes\vv^*(T,c)\cup I\otimes\vv(T,c)&\textrm{if the root is uncoloured, unary},\\\KK_2(n)\otimes\vv^*(T,c)\cup\KK_1(n)\otimes\vv(T,c)&\textrm{if the root is coloured},\end{cases}$$where we have abbreviated\begin{align*}\vv^*(T,c)&=\bigcup_{i=1}^n\uu(T_1,c_1)\otimes\cdots\otimes\uu^*(T_i,c_i)\otimes\cdots\otimes\uu(T_n,c_n),\\\vv(T,c)&=\uu(T_1,c_1)\otimes\cdots\otimes\uu(T_n,c_n).\end{align*}

One now proves by induction that $\uu^*(T,c)\ito\uu(T,c)$ is an $\Aut(T,c)$-cofibration, exactly as in the proof of \cite[Lemma 5.9]{BM}; towards the end, the case where the root is uncoloured splits into two subcases: the one where the root is not unary is as in \cite{BM}, while for the one where the root is unary, the map $\uu^*(T,c)\ito\uu(T,c)$ is of the form $I\otimes B\cup_{I\otimes A}\KK_1(1)\otimes A\ito\KK_1(1)\otimes B$, and \cite[Lemma 5.10]{BM} applies again. (Note that a ``$G$'' is missing in line 3 of Lemma 5.10, compare \cite[Lemma 2.5.3]{BM2} and the remark following it).\end{proof}

\vspace{5ex}

\noindent{\sc Universit\'e de Nice, Laboratoire J.-A. Dieudonn\'e, Parc Valrose, 06108 Nice Cedex, France.}\hspace{2em}\emph{E-mail:} cberger$@$math.unice.fr\vspace{2ex}

\noindent{\sc Mathematisch Instituut, Postbus 80.010, 3508 TA Utrecht, The Ne-therlands.}\hspace{2em}\emph{E-mail:} moerdijk$@$math.uu.nl

\end{document}